\documentclass[psamsfonts]{amsart}

\usepackage{amssymb,amsfonts}
\usepackage[all,arc]{xy}
\usepackage{enumerate}
\usepackage{mathrsfs}
\usepackage{amsmath}
\usepackage{fullpage}
\usepackage{hyperref}
\usepackage{graphicx}

\newtheorem{theorem}{Theorem}

\newtheorem{claim}{Claim}

\newtheorem{example}{Example}

\newtheorem{lemma}{Lemma}

\newtheorem{remark}{Remark}

\numberwithin{equation}{section}

\title{The Gibbons-Hawking ansatz over a wedge}
\author{Martin de Borbon}

\begin{document}

\begin{abstract}
	We  discuss the Ricci-flat `model metrics' on $\mathbf{C}^2$ with cone singularities along the conic \( \{zw=1\} \) constructed by Donaldson -\cite{DonaldsonKMCS}, Section 5- using the Gibbons-Hawking ansatz over wedges in \(\mathbf{R}^3\). In particular we describe their asymptotic behavior at infinity and compute their energies. 	
\end{abstract}	
\maketitle

\section{Introduction}

Fix \(0<\beta<1\). Let \(W\) be a wedge in \(\mathbf{R}^3\) of angle \(2\pi\beta\) delimited by two planes. Let \(p\) be a point in the interior of the wedge which is equidistant from the two faces and is located at distance \(1\) from the edge of \(W\). Take \(f\) to be the associated Green's function for the Laplacian with pole at \(p\) and zero normal derivative at the boundary of \(W\). The Gibbons-Hawking ansatz produces an hyperk\"ahler 4-manifold with boundary \(\overline{P}\) endowed with a circle action which preserves all the structure and has the closure of  \(W\) as its space of orbits. The rotation that takes one face of \(W\) to the other is lifted to an isometry \(F\) of \(\overline{P}\), fixing the points over the edge of \(W\). We identify points on the boundary of \(\overline{P}\) which correspond under \(F\) to obtain a smooth manifold \(P\) without boundary, endowed with a metric \(g_{RF}\) which has cone angle \(2\pi\beta\) in transverse directions to the points fixed by \(F\). The upshot is that the direction of the edge of \(W\) defines a global complex structure \(I\) on \(P\) with respect to which \(g_{RF}\) is K\"ahler; and the complex manifold \((P, I)\) is indeed a very familiar one.

\begin{theorem} \label{THEOREM}
	
\(g_{RF}\) defines a Ricci-flat K\"ahler metric on \(\mathbf{C}^2\); it is invariant under the \(S^1\)-action \( e^{i\theta} (z, w)= (e^{i\theta}z, e^{-i\theta}w) \), it has cone angle \(2 \pi \beta \) along the conic \( C= \{zw=1\} \) and its volume form is

\begin{equation*}
	\mbox{Vol}(g_{RF})= (\beta^2/2) |1-zw|^{2\beta -2} \Omega \wedge \overline{\Omega},
\end{equation*}
where \( \Omega = (1/ \sqrt{2}) dzdw \). Moreover

\begin{enumerate}
	\item \label{Item 1} At points on \(C\) it has cone singularities in a \(C^{\alpha}\) sense -as defined in \cite{DonaldsonKMCS}- with H\"older exponent \( \alpha =1 \) if \( 0< \beta \leq 1/2 \) and \( \alpha = (1/ \beta) -1 \) if \( 1/2< \beta <1 \)
	\item \label{Item 2} It is asymptotic to the Riemannian cone \(\mathbf{C}_{\beta} \times \mathbf{C}_{\beta} \) at rate \(-4\) if \( 0 < \beta \leq 1/2 \) and \(-2/ \beta \) if \( 1/2 < \beta <1 \)
	\item \label{Item 3} Its energy is finite and is given by
\begin{equation*}
	E(g_{RF}) = 1 - \beta^2
\end{equation*}

\end{enumerate}
	
\end{theorem}	

We collect some background material on the Green's function of a wedge and the Gibbons-Hawking ansatz in Section \ref{background}.  The proof of Theorem \ref{THEOREM} is done in Section \ref{proof thm section}; the identification of \((P, I)\) with \(\mathbf{C}^2\) is already in \cite{DonaldsonKMCS}, we repeat the argument filling-in small details. The original content of this article rests on the three items \ref{Item 1}, \ref{Item 2} and \ref{Item 3}, which  are proved in \ref{complex structure section}, \ref{asymptotics section} and \ref{energy section} respectively. Finally, in Section \ref{additiona comments}, we discuss the sectional curvature of the metrics \(g_{RF}\) and the limits when \(\beta \to 0 \).

The interest in Theorem \ref{THEOREM} comes from the blow-up analysis of the K\"ahler-Einstein (KE) equations in the context of solutions with cone singularities. 
In the case of smooth KE metrics on complex surfaces the solutions can only degenerate -in the non-collapsed regime- by developing isolated orbifold points, and the blow-up limits at these are the well-known ALE spaces. In the conical case a new feature arises when the curves along which the metrics have singularities degenerate. In this setting, the \(g_{RF}\) furnish a model for blow-up limits at a point where a sequence of smooth curves develops an ordinary double point. Models for blow-up limits of sequences in which the curves develop an ordinary \(d\)-tuple point are constructed in \cite{martin}.

The energy of a Riemannian manifold \((M, g)\) is defined as 
\[ E(g) = \frac{1}{8\pi^2} \int_M | \mbox{Rm}(g)|^2 dV_g , \]
where \(\mbox{Rm}(g)\) denotes the curvature operator of the metric. In our case \(g_{RF}\) is smooth on \(\mathbf{C}^2 \setminus C\) and we integrate on this region. Following next we clarify the meaning of the first two items in Theorem \ref{THEOREM}, but first let us introduce some notation. 

We write \( \mathbf{C}_{\beta}  \) for the complex numbers endowed with the singular metric \( \beta^2 |\xi|^{2\beta -2} |d\xi|^2 \). 
We recognize it as  the standard cone of total angle \(2\pi\beta\). Indeed, if we introduce the `cone coordinates'
\begin{equation}
\xi = r^{1/\beta} e^{i\theta} 
\end{equation}
then \(\beta^2 |\xi|^{2\beta -2} |d\xi|^2 = dr^2 + \beta^2 r^2 d\theta^2 .\) There are two flat model metrics which are relevant to us. The first is  \(\mathbf{C}_{\beta} \times\mathbf{C}\), which captures the local behavior of \(g_{RF}\) at points on the conic. In complex coordinates 
\(g_{loc} =  \beta^2 |z_1|^{2\beta -2} |dz_1|^2  + |dz_2|^2 . \)
Secondly is  \(\mathbf{C}_{\beta} \times\mathbf{C}_{\beta}\), which models the asymptotic behavior of \(g_{RF}\) at infinity. In complex coordinates \( g_F = \beta^2 |u|^{2\beta -2} |du|^2 + \beta^2 |v|^{2\beta -2} |dv|^2 .\) We introduce `spherical coordinates' by setting
\[ \rho^2 = |u|^{2\beta} + |v|^{2\beta} . \]
The function \(\rho\) measures the intrinsic distance to \(0\) and it is easy to check to that \(g_F = d\rho^2 + \rho^2 \overline{g} \) where
\(\overline{g}\) is a metric on the 3-sphere with cone angle \(2\pi\beta\) along the Hopf circles determined by the complex lines \(\{u=0\}\) and \(\{v=0\}\). We proceed  with the explanation of Theorem \ref{THEOREM}

\begin{itemize}
	
	\item  
	
	Item \ref{Item 1}. Let \(p \in C \) and  \( (z_1, z_2) \) be  complex coordinates centered at \(p\) such that \(C=\{z_1=0\}\). Let \(z_1=r_1^{1/\beta}e^{i\theta_1}\). Set \(\epsilon_1 = dr_1 + i\beta r_1 d\theta_1 \) and \(\epsilon_2=dw\). 
	The K\"ahler from associated to \(g_{RF}\) writes as
	\[ \omega_{RF} = i\sum_{j,k} a_{j\overline{k}} \epsilon_j \overline{\epsilon_k} \]
	for  smooth functions \(a_{j\overline{k}}\) on the complement of \(\{z_1 =0\}\).
	We say that $g_{RF}$ is \(C^{\alpha}\)
	if, for every $p \in C$ and holomorphic coordinates centered at \(p\) as above, the  $a_{j\overline{k}}$ extend to $\{z_1=0\}$ as \(C^{\alpha}\) functions in the cone coordinates \( (r_1 e^{i \theta_1}, z_2 ) \). We also require the matrix $(a_{j\overline{k}} (p))$ to be positive definite  and that  $a_{1\overline{2}} =0$ when \(z_1=0\). In particular these conditions imply that \( A^{-1}g_{loc} \leq g_{RF} \leq A g_{loc} \) for some \( A >0 \).
	
	\item Item \ref{Item 2}.
	We say that \(g_{RF}\) is asymptotic to \(g_F\) at rate \(-\mu \), for some \( \mu>0\), if there is a closed ball \(B \subset \mathbf{C}^2 \) and a map \(\Phi : \mathbf{C}^2 \setminus B \to \mathbf{C}^2 \) which is a diffeomorphism onto its image and with the property that
	\[ | \Phi^{*}g_{RF} - g_F |_{g_F} \leq A \rho^{-\mu}, \hspace{3mm} | \Phi^{*}I - I |_{g_F} \leq A \rho^{-\mu} \] 
	for some constant \(A>0\); where \(I\) denotes the standard complex structure of \(\mathbf{C}^2 \). We write \( (z, w)= \Phi (u,v) \), so that necessarily \( \Psi (\{uv=0\}) \subset \{zw=1\} \).

\end{itemize}
 
\subsection*{Acknowledgments} I learned abou the Gibbons-Hawking anstaz from my supervisor, Simon Donaldson, back in 2011 during the first year of my PhD. I would like to thank Simon for his beautiful teaching and the European Research Council Grant 247331 for financial support.
This article was written during my visit to the 2017 Summer Program at IMPA, Rio de Janeiro. I would also like to thank IMPA for providing me with excellent working conditions.
	
\section{Background} \label{background}

\subsection{Potential theory on a wedge} \label{green function section}

Consider the wedge

\[ W = \{(re^{i\tilde{\theta}}, s) \in \mathbf{R}^3 \hspace{3mm} \mbox{s.t.} \hspace{2mm} -\pi \beta < \tilde{\theta} < \pi \beta   \} . \]
Write \(S=\{0\} \times \mathbf{R}\) for the edge of \(W\) and let \(p=(1,0,0)\). It is a fact that there is a unique continuous and positive function \( \tilde{\Gamma}_p : \overline{W} \setminus \{p\} \to \mathbf{R}_{>0} \) which tends to \(0\) at infinity and solves the boundary value problem
$$\left\{\begin{array}{cc}
\displaystyle \Delta \tilde{\Gamma}_p= \delta_p & \mbox{ on } W, \\
\\
\displaystyle \frac{\partial \tilde{\Gamma}_p}{\partial\nu}=0 & \mbox{ on }  \partial W . \\
\end{array}\right.$$
The first equation is meant to be interpreted in the sense of distributions, \(\Delta\) is the standard Euclidean Laplacian and  \(\delta_p \) is the Dirac delta at \(p\). In the second equation \(\nu\) denotes the outward unit normal vector, which is well defined on the complement of the edge. In other words,  \(\tilde{\Gamma}_p\) is the Green's function for the Laplace operator with pole at \(p\) associated to the Neumann boundary value problem over \(W\). 

Standard elliptic regularity theory -Weyl's lemma- implies that \(\tilde{\Gamma}_p\) is smooth on \( W \setminus \{p\} \). Indeed, on \(W\) we can write

\begin{equation} \label{representation green}
\tilde{\Gamma}_p (x) = \frac{1}{4\pi|x-p|} + F	
\end{equation}
for some smooth harmonic function \(F\). The behavior of \(\tilde{\Gamma}_p\) around points on the edge is more subtle.

We use a rotation to identify the faces of \(W\). Write \(\tilde{\theta}=\beta \theta\). We are led with
a metric on \(\mathbf{R}^3\) with cone angle \(2\pi\beta\) along \(S=\{0\}\times \mathbf{R}\),
\begin{equation}
g_{\beta} = dr^2 + \beta^2 r^2 d\theta^2 + ds^2 .
\end{equation}
Write \(\Delta_{\beta}\) for its Laplacian.
We let \(\Gamma_p(re^{i\theta}, s)= \tilde{\Gamma}_p (re^{i\tilde{\theta}}, s)\). The function \(\Gamma_p\) is continuous on \(\mathbf{R}^3 \setminus\{p\}\), smooth on the complement of \(S \cup\{p\}\) and solves the distributional equation
\[ \Delta_{\beta} \Gamma_p = \delta_p . \]
It is shown in \cite{DonaldsonKMCS} that \(\Gamma_p\) is \(\beta\)-smooth at  points of \(S\); which means that it is a smooth function of the variables \(r^{1/\beta} e^{i\theta}, r^2, s \). Moreover, we have a  `polyhomogeneous' expansion

\begin{equation} \label{expansion}
\Gamma_p = \sum_{j, k \geq 0} a_{j, k} (s) r^{(k/\beta) + 2j} \cos (k\theta) 
\end{equation}
with $a_{j, k}$ smooth functions of \(s\) and which converges uniformly when \(r \leq 1/4\).

Allowing the point \(p\) to vary we obtain, in the usaul way, the function \( G(p, q)= \Gamma_p (q) \) which provides an inverse for \(\Delta_{\beta} \psi = \phi\) by letting \(\psi(x)= \int G(x, y) \phi (y) dV_{\beta}(y) \). The symmetries and dilations of \(g_{\beta}\) are reflected in that
\[ G(T_l p, T_l q) = G(p, q), \hspace{2mm} G(R_{\gamma}p, R_{\gamma}q)= G(p, q), \hspace{2mm}  G( m_{\lambda} p, m_{\lambda} q) = \lambda^{-1} G(p, q) \] 
where \( T_l(r, \theta, s) = (r, \theta, s+l) \), \(R_{\gamma} (r, \theta, s) = (r, \theta + \gamma, s) \) and  \( m_{\lambda} (r, \theta, s) = (\lambda r, \theta, \lambda s) \) for \(\lambda>0\).
There is also the symmetric property \(G(p, q)= G(q, p)\). 
It follows from the  \(\beta\)-smoothness that there is \(\kappa>0\) such that
\begin{enumerate}
	\item 
	\(|G(0, p)| \leq \kappa \) for every \(p\) with \(|p|=1\)
	\item
	\(|G(x_1, p)-G(x_2, p)| \leq \kappa |x_1 - x_2|^{1/\beta} \) whenever \(|p|=1\) and \(|x_1|, |x_2| \leq 1/2\)
\end{enumerate}	
It is easy to write the Green's function with the pole located at \(S\),
\[ G (0, x)=\frac{1}{4 \pi \beta |x|} . \]
By homogeneity, if \(|x| \geq 2 |p| \)
\begin{equation}
	|G(x, p)- G(x, 0)| = |x|^{-1} |G(|p||x|^{-1}, x|x|^{-1}) - G(0, |x|^{-1}x)| \leq \kappa |x|^{-1-1/\beta} .
\end{equation}
In particular we see that \(\Gamma_p\) decays as \(|x|^{-1}\). 

We include an observation regarding formula \ref{representation green} which will be useful for us later on

\begin{lemma} \label{positive lemma}
	\(F>0\)
\end{lemma}

\begin{proof}
	Since \(F\) is harmonic it is enough to show that it is positive on \(\partial(B \cap W )\) for any sufficiently large ball \(B\).  Since \( \tilde{\Gamma}_p \) is asymptotic to \(1/4\pi\beta|x|\) it follows that \(F>0\) on \(\partial B \cap W \). Note that \(F\) restricted to the edge is equal to \( (\beta^{-1}-1)/ 4\pi |x-p| >0 \). The fact that the normal derivative of \(\tilde{\Gamma}_p\) is zero at the boundary of \(W\) implies that \(F\) has no critical points when restricted to these planes, it then follows that \(F>0\) on \(\partial W \cap B\).
\end{proof}

To finish this section and for the sake of completeness we comment a bit more on the expansion \ref{expansion}. The coefficients $a_{j, k}$ are given in terms of Bessel's functions and we want to indicate how these arise. The technique is separation of variables. We write
\begin{equation} \label{separation var}
	G(r, \theta, s; r', \theta', s') = \sum_{k=0}^{\infty} G_k(r, r', R) \cos k(\theta-\theta')
\end{equation}
with \(R=|s-s'|\). We decompose
\begin{equation*}
\triangle_{\beta} = \frac{\partial^2}{\partial r^2} + \frac{1}{r}\frac{\partial}{\partial r} + \frac{1}{\beta^2 r^2} \frac{\partial^2}{\partial \theta^2} + \frac{\partial^2}{\partial s^2} = L + \frac{\partial^2}{\partial s^2} .
\end{equation*}
The point is that -for any integer \(k \geq 0 \) and \(\lambda \geq 0 \)- the function \(\phi= J_{\nu}(\lambda r)e^{ik\theta} \) is an eigenfunction for \(L\) with \(L\phi = - \lambda^2 \phi \); where \(\nu= k/\beta \) and
\begin{equation} \label{Bessel funct}
	J_{\nu}= \sum_{j=0}^{\infty} (-1)^j \frac{(z/2)^{\nu + 2j}}{j! (\nu+j)!}
\end{equation}
is Bessel's function, which solves \(f^{''} + z^{-1}f' + (1-\nu^2 z^{-2})f =0 \). This leads to a formula for the heat kernel associated to \(\Delta_{\beta}\) 
\[ (4\pi t)^{-1/2} e^{-R^2/4t} \sum_{k=0}^{\infty} \left( \pi^{-1} \int_{0}^{\infty} e^{-\lambda^2 t} J_{\nu}(\lambda r) J_{\nu}(\lambda r') d\lambda \right) \cos k(\theta - \theta') . \]
The Green's function is obtained by integration of the heat kernel with respect to the time parameter and this gives us
\begin{equation} \label{formal expresion}
	G_k (r, r', R)= \int_{0}^{\infty} \int_{0}^{\infty} (4\pi t)^{-1/2} e^{-\lambda^2 t -R^2/4t} J_{\nu}(\lambda r) J_{\nu}(\lambda r') d\lambda dt .
\end{equation}
We fix \((r', \theta', s')=(1, 0, 0)\), replace \ref{formal expresion} into \ref{separation var}, expand the Bessel's functions into power series \ref{Bessel funct} and exchange the integral with the summation; to obtain a formal polyhomogeneous expansion as  in \ref{expansion}. The validity of the expression is guaranteed provided we check uniform convergence. In order to do this the integral \ref{formal expresion} has  to be properly manipulated, suitable bounds must be derived and some expertise with Bessel's functions is required -see \cite{DonaldsonKMCS}-.

\subsection{The Gibbons-Hawking ansatz} \label{GH section}
This well-known construction provides a `local' correspondence between positive harmonic functions on domains in \( \mathbf{R}^3 \) and hyperk\"ahler structures with \(S^1\) symmetry. More precisely, let \(x_1, x_2, x_3\) be standard coordinates on \(\mathbf{R}^3\) and let \(f\) be a positive harmonic function on \(\Omega\). Consider an \(S^1\)-bundle over \(\Omega\) equipped with a connection \footnote{By a connection we mean an \(S^1\)-invariant \(1\)-form on the total space which gives \(1\) when contracted with the derivative of the \(S^1\)-action. Its curvature is \(d\alpha \) and it is a general fact that it is the pull-back by the bundle projection of a closed \(2\)-form on the base whose de Rham cohomology class represents \(-2\pi c_1\). We shall often suppress the pull-back by the bundle projection in our formulas.} \(\alpha\) which satisfies the Bogomolony equation 
\begin{equation} \label{Bogomolony}
	d \alpha = - \star df . 
\end{equation}
The hyperk\"ahler structure is then defined by means of the three 2-forms
\begin{equation}
	\omega_i = \alpha dx_i + f dx_j dx_k , 
\end{equation}
here and in the rest of the article we use the notation of the indices \((i, j, k)\) varying over the cyclic permutations of \((1, 2, 3)\). The Bogomolony equation \ref{Bogomolony} is indeed equivalent to the \(\omega_i\) being closed. The bundle projection is then characterized as th hyperk\"ahler moment map for the \(S^1\)-action.   

\begin{example} \label{euclidean metric}
	A basic case is that of the Euclidean metric on \(\mathbf{R}^4 \cong \mathbf{C}^2\) equipped with the \(S^1\)-action \(e^{i\theta}(z_1, z_2)=(e^{i\theta}z_1, e^{-i\theta}z_2)\) which preserves its standard hyperk\"ahler structure. The hyperk\"ahler moment map agrees with the Hopf map
	\begin{equation} \label{Hopf map}
	H(z_1, z_2)= \left( z_1 z_2, \frac{|z_1|^2-|z_2|^2}{2} \right) .
	\end{equation}
	Removing \(0\) gives an \(S^1\)-bundle with first Chern class equal to \(-1\); and it is straightforward to check that
	\begin{equation}
	f= \frac{1}{2|x|}, \hspace{3mm} \alpha = \mbox{Re} \left(  i \frac{\overline{z_2}dz_2 - \overline{z_1}dz_1}{|z_1|^2 + |z_2|^2}\right) .
	\end{equation}
\end{example}

A unit vector \(v\) in \(\mathbf{R}^3\) determines a parallel complex structure on the hyperk\"ahler manifold, by sending the horizontal lift of the constant vector field \(v\) to the derivative of the \(S^1\)-action. We now review an explicit construction of holomorphic functions for these complex structures on a particular case which will be relevant for us, our reference is Section 5.3 in \cite{DonaldsonKMCS}.

Assume that \(\Omega\) is the product of a \emph{simply connected} domain \(U\subset \mathbf{R}^2\) with the \(x_3\) axis. We will consider the complex structure determined by the \(x_3\) direction. We trivialize the bundle and denote the circle coordinate with \(e^{it}\), so that
\[ \alpha = dt + \sum_{j=1}^3 a_j dx_j . \]
We can change gauge by \( \tilde{t}= t - \int_0^{x_3} a_3(x_1, x_2, q) dq \) -or in other words parallel translate in the \(x_3\) direction- and assume that \(a_3 \equiv 0\). The Bogomolony equation \ref{Bogomolony} amounts to 
\begin{equation} \label{bogomolony local}
	\frac{\partial a_2}{\partial x_3} = f_1, \hspace{3mm} \frac{\partial a_1}{\partial x_3}= -f_2, \hspace{3mm} \frac{\partial a_1}{\partial x_2} - \frac{\partial a_2}{\partial x_1} = f_3 ;
\end{equation}
where \(f_i\) means \(\partial f / \partial x_i \). We further assume that \( f_3 =0 \) on \( \{x_3 =0\} \), in other words this is to say that \(\alpha\) restricts to a flat connection over \(U\). Since \(U\) is simply connected we can perform a gauge transformation \( \tilde{t} = t - \phi (x_1, x_2) \) and assume that \( a_1 = a_2 = 0 \) on the slice $\{x_3=0\}$. 

The horizontal lifts of the coordinate vectors \(\partial/\partial x_i\) are given by

$$ \tilde{\frac{\partial}{\partial x_1}}=  \frac{\partial}{\partial x_1} - a_1  \frac{\partial}{\partial t}, \hspace{3mm} \tilde{\frac{\partial}{\partial x_2}}=  \frac{\partial}{\partial x_2} - a_2  \frac{\partial}{\partial t} , \hspace{3mm} \tilde{\frac{\partial}{\partial x_3}}= \frac{\partial}{\partial x_3} ; $$
and the complex structure is defined as
\begin{equation}
I  \tilde{\frac{\partial}{\partial x_1}} =   \tilde{\frac{\partial}{\partial x_2}}, \hspace{3mm} I \frac{\partial}{\partial x_3} = -f  \frac{\partial}{\partial t}
\end{equation}
The Cauchy-Riemann equations for a function $h$ to be holomorphic w.r.t. $I$ are then given by

\begin{equation} \label{Cauchy Riemann}
\frac{\partial h}{\partial x_1} + i  \frac{\partial h}{\partial x_2} = (a_1 + i  a_2) \frac{\partial h}{\partial t}, \hspace{4mm} \frac{\partial h}{\partial x_3}= if \frac{\partial h}{\partial t}
\end{equation}
We look for a function $h$ which has weight one for the circle action, so that \( {\partial h}/{\partial t} = ih \) . We use separation of variables and write $h = \tilde{h} e^{it}$ with $\tilde{h} = \tilde{h} (x_1, x_2, x_3)$. The second equation in \ref{Cauchy Riemann} gives us $\partial \tilde{h}/ \partial x_3 = -f \tilde{h}$; so that

\begin{equation} \label{hol1}
h= h_0 e^{-u} e^{it}, \hspace{3mm} u = \int_0^{x_3} f(x_1, x_2, q) dq, \hspace{3mm} h_0 = h_0(x_1, x_2)
\end{equation}
Recall that $a_1 = a_2 = 0$ on the slice $\lbrace x_3=0 \rbrace$. Let $h_0$ be any solution of the equation

\begin{equation} \label{hol2}
\frac{\partial h_0}{\partial x_1} + i \frac{\partial h_0}{\partial x_2} = 0 ,
\end{equation}
in other words an holomorphic function of $x_1 + i x_2$.

\begin{lemma} \label{hol function lemma}
	The function $h$ defined by \ref{hol1} and \ref{hol2} solves \ref{Cauchy Riemann}. 
\end{lemma}

\begin{proof}
The Cauchy-Riemann equations \ref{Cauchy Riemann} is an over-determined system. It follows from the definition of $h$ that it solves the second equation in \ref{Cauchy Riemann} and that it solves the first equation in \ref{Cauchy Riemann} only in the slice \(\{x_1 =0\} \). The point is that \(h\) is a solution thanks to the integrability condition provided by the Bogomolony equation \ref{Bogomolony}, as the following computation shows
$$\frac{\partial h}{\partial x_1} + i \frac{\partial h}{\partial x_2} = e^{-u}e^{it} \left( \frac{\partial h_0}{\partial x_1} + i \frac{\partial h_0}{\partial x_2} -  h_0 \frac{\partial u}{\partial x_1} - h_0   i \frac{\partial u}{\partial x_2} \right) = i h \left(- \frac{\partial u}{\partial x_2} + i \frac{\partial u}{\partial x_1} \right)$$ 
and it follows from \ref{bogomolony local} that

 \[ \frac{\partial u}{\partial x_1} = a_2 , \hspace{3mm} \frac{\partial u}{\partial x_2} = -a_1  \]
\end{proof}

Similarly, $h_0 e^u e^{-i\psi}$ defines a holomorphic function with weight $-1$ for the circle action. Given any holomorphic function of the complex variable \(x_1 + i x_3\), \(h_0\) on \(U\), the pair  
\begin{equation}
z= h_0 e^{-u}e^{i\psi}, \hspace{4mm} w=h_0 e^u e^{-i\psi}
\end{equation}
defines an \(S^1\)-equivariant   holomorphic map from \( (\Omega \times S^1, I) \) to a domain in \(\mathbf{C}^2 \) equipped with the circle action \(e^{it}(z, w)= (e^{it}z, e^{-it}w)\).

\begin{example}
	Taub-Nut metric. Let \(c>0\) and consider the harmonic function 
	\[f = 2c + \frac{1}{2|x|} . \]
	Let \(\alpha\) be the connection on the Hopf bundle \( H: \mathbf{R}^4 \setminus \{0\} \to \mathbf{R}^3 \setminus \{0\} \) given in Example \ref{euclidean metric}, so that \( d\alpha = - \star df \). We look at complex structure \(I\) on \(\mathbf{R}^4\) determined by the \(x_3\)-axis. \footnote{On \( \mathbf{R}^4 \) we have standard coordinates \(s_1, s_2, s_3, s_4 \). For notational convenience we write \(z_1=s_1 + is_2\) and \(z_2 = s_3 + is_4\), but \(z_1, z_2\) are \emph{not} necessarily complex coordinates for \( (\mathbf{R}^4, I)\).} 
	
	We want a suitable trivialization of the bundle which fits us in the context of Lemma \ref{hol function lemma}. Note that it follows from \ref{Hopf map} that 
	\[ 2|x| = |z_1|^2 + |z_2|^2, \]
	so \( |z_1| = (|x| - x_3)^{1/2} \) and \(|z_2| =(|x| + x_3)^{1/2} \). Write \( \xi= x_1 + i x_2 = |\xi| e^{i\theta} \) and let \(U\) be the complement of the negative real axis in the \( \xi \)-plane so that \(-\pi < \theta  < \pi\). Over \( \Omega \) we have the following trivialization
	\[ \Phi (x, e^{it})=(z_1, z_2)= \left( (|x|-x_3)^{1/2}e^{i\theta/2}e^{it}, (|x| + x_3)^{1/2}e^{i\theta/2}e^{-it} \right)  \]
	and it is easy to check that
	\[ \alpha = dt - \frac{x_3}{2|x|}d\theta , \]
	which clearly satisfies \(a_3 \equiv 0\) and \(a_1 = a_2 = 0\) when \(x_3=0\). \footnote{The restriction of the Hopf connection to the punctured plane $\{x_3 =0\}\setminus \{(0,0)\}$ is a flat connection with \emph{non-trivial} holonomy, this forces us to introduce a cut in the plane in order to define the desired trivialization with \(a_1=a_2=0\).}
	
	We apply Lemma \ref{hol function lemma} with \(h_0 = \sqrt{\xi}\). We can easily compute
	
	\[ u= \int_0^{x_3} 2c + \frac{1}{2 \sqrt{|\xi|^2 + q^2}} dq = 2cx_3 + \frac{1}{2} \log \left(  x_3 + \sqrt{x_3^2 + |\xi|^2} \right)  -\frac{1}{2} \log |\xi| \]
	and therefore 
	\begin{equation} \label{Taub Nut coordinates}
		 z= e^{c(|z_1|^2 - |z_2|^2)} z_1, \hspace{3mm} w= e^{c(|z_2|^2 -|z_1|^2)}z_2 .
	\end{equation}
	It is then clear that the map \( (z, w) = (z(z_1, z_2), w(z_1, z_2)) \) extends to define a biholomorphism between \( (\mathbf{R}^4, I) \) and \(\mathbf{C}^2\). 
	
	As a final remark we relate the diffeomorphism \ref{Taub Nut coordinates} to LeBrun's expression for the K\"ahler potential of the Taub-Nut metric -see \cite{LeBrunTaubNut}-. Indeed, \( \omega = \alpha dx_3 + f dx_1dx_2 \) and it is easy to check that \( \omega = i \partial \overline{\partial} (|x| + c (x_1^2 + x_2^2 + 2x_3^2))\). Since \(|x|= (1/2)(|z_1|^2 + |z_2|^2) \) and \( x_1^2 + x_2^2 + 2 x_3^2 = (1/2) (|z_1|^4 + |z_2|^4) \) we obtain
	\[ \omega = \frac{i}{2} \partial \overline{\partial} \left(  |z_1|^2 + |z_2|^2 + c(|z_1|^4 + |z_2|^4) \right)  , \]
	with \(|z_1|, |z_2|\) determined implicitly in terms of \(z, w\) by means of  \ref{Taub Nut coordinates}.
\end{example}

\subsection{4-dimensional Riemannian geometry} \label{4d section}
Let \((M, g)\) be an oriented Riemannian four-manifold and let \(\mbox{Rm}(g) : \Lambda^2 \to \Lambda^2 \) be its curvature operator. The decomposition \(\Lambda = \Lambda^{+} \oplus \Lambda^{-} \) of the 2-forms into self and anti-self-dual given by the Hodge star operator of \(g\) determines the well-known decomposition of \( \mbox{Rm}(g) \) into four three by three blocks 
\[
\left( 
\begin{array}{c|c}
\frac{s}{12} + W^{+} & \mathring{r} \\
\hline
\mathring{r} & \frac{s}{12} + W^{-} 
\end{array}
\right) .
\]
These blocks can also be interpreted in terms of the curvature \(F_{\nabla}\) of the Levi-Civita connection on the bundles \( \Lambda^{+}, \Lambda^{-} \). Indeed, if \( \theta_1, \theta_2, \theta_3\) is an orthonormal triple of anti-self-dual forms we can write
\[ F_{\nabla} (\theta_i) = F_j \otimes \theta_k - F_k \otimes \theta_j \]  
for some 2-forms \(F_1, F_2, F_3\). We write the anti-self-dual parts as \(F_{i}^{-} = \sum_{j=1}^{3} c_{ij} \theta_j \). The fact is that \( (c_{ij})_{1 \leq i, j \leq 3} \) agrees with the block \( s/12 + W^{-} \); and similarly for the other blocks.  The curvature 2-forms are  given by \(F_i = T_i - T_j \wedge T_k\) where \(T_1, T_2, T_3 \) are the connection 1-forms \( \nabla (\theta_i) = T_j \otimes \theta_k - T_k \otimes \theta_j \). 

The torsion-free property gives us
\begin{equation} \label{cartan}
	d\theta_i = T_j \wedge \theta_k - T_k \wedge \theta_j .
\end{equation}
A simple algebraic fact -see Proposition 2.3 in \cite{FineGauge}- is that the system of equations \ref{cartan} for \(T_1, T_2, T_3\) has the unique solution 
\begin{equation} \label{formula connection}
	2T_i = \star \psi_i + \star (\star \psi_j \wedge \theta_k) - \star (\star \psi_k \wedge \theta_j) , 
\end{equation}
where \(\psi_i = d\theta_i \) and \(\star\) denotes the Hodge operator of \(g\). This fact can be interpreted as a characterization of the Levi-Civita connection on \(\Lambda^{-}\) as the unique which is metric and torsion-free; somewhat analogous to the Cartan's lemma.

We use the previous discussion to compute the energy distribution \(|\mbox{Rm}(g)|^2 \) of a metric \(g= f dx^2 + f^{-1} \alpha^2\) given by the Gibbons-Hawking anstaz. There is an orthonormal frame of self-dual 2-forms \( \omega_i = \alpha dx_i + f dx_j dx_k \); since \(d \omega_i =0\) the only non-vanishing part of the curvature operator is \(W^{-}\), the anti-self-dual Weyl curvature tensor and it is a general fact that \(W^{-}\) is symmetric and trace-free. We consider the orthonormal frame of \(\Lambda^{-}\) given by \(\theta_i = \alpha dx_i - f dx_j dx_k \); so \(d \theta_i = -2f_i dx_1 dx_2 dx_3 \) and \ref{formula connection} gives us
\[ T_i = \frac{f_i}{f^2}\alpha - \frac{f_j dx_k -f_k dx_j}{f} . \]
We compute the curvature forms and express their anti-self-dual components with respect to the \(\theta_i\) frame to obtain
\[ c_{ij} = \frac{f_{ij}}{f^2} - 3 \frac{f_i f_j}{f^3} + \delta_{ij} \frac{|Df|^2}{f^3} .\]
We can write this more succinctly as
\begin{equation} \label{curv op GH met}
	 W^{-} = (-f/2) \mathring{\mbox{Hess}} (f^{-2}) ,
\end{equation}
where \(\mathring{\mbox{Hess}}\) denotes the standard Euclidean trace-free Hessian.

In order to achieve our goal we proceed by straightforward computation,
\[ |\mbox{Rm}(g)|^2 = \sum_{i,j} c_{ij}^2 = 6 f^{-6} |Df|^4 + f^{-4} \sum_{i,j} f_{ij}^2 - 6 f^{-5} \sum_{i,j} f_{ij}f_if_j . \]
Let \(\Delta\) be the Euclidean Laplacian, so that \(\Delta f =0 \). It follows easily that

\[ \Delta f^{-1} = -2f^{-3} |Df|^2, \hspace{3mm} (1/2) \Delta \Delta f^{-1} = 12 f^{-5} |Df|^4 + 2 f^{-3} \sum_{i,j} f_{ij}^2 - 12 f^{-4} \sum_{i,j} f_{ij}f_if_j . \]
Comparing we obtain the formula -see Remark 2.4 in \cite{GrossWilson}-
\begin{equation} \label{energy distribution}
	|\mbox{Rm}(g)|^2 = \frac{1}{4 f} \Delta \Delta f^{-1} .
\end{equation} 

\section{Proof of Theorem} \label{proof thm section}

We go back to the construction of the metric \(g_{RF}\) mentioned in the Introduction. It is convenient to perform the gluing right in the beginning. Equivalently, we start with

\begin{equation}
g_{\beta} = dr^2 + \beta^2 r^2 d\theta^2 + ds^2.
\end{equation} 
Write \(\star_{\beta}\) for its Hodge operator, which acts one \(1\)-forms as
\begin{equation*}
\star_{\beta} dr = \beta  r d\theta ds, \hspace{3mm} \star_{\beta} \beta r d\theta =  ds dr, \hspace{3mm} \star_{\beta} ds = \beta  r drd\theta .
\end{equation*} 

Let $p=(1,0, 0)$ and let $\Gamma_p$ be the Green's function for \(\Delta_{\beta}\) with pole at \(p\) -see Subsection \ref{green function section}-. We take \(f=2\pi\Gamma_p \), so that \( -\star_{\beta} df \) integrates \(2\pi\) over spheres centered at \(p\). Note that 
\begin{equation*}
	 d \star_{\beta} df = (\triangle_{\beta} f) V_{\beta} =0
\end{equation*}
where \(dV_{\beta} = \beta r dr d\theta ds \) denotes the volume form.
Let \( \Pi : P_0 \to \mathbf{R}^3 \setminus (S \cup \{p\}) \) be the \(S^1\)-bundle with \( c_1 (P_0)=-1 \). We shall show that there is a connection  \(\alpha\) on \(P_0\) with curvature \(-\star_{\beta} df \) and with trivial holonomy along small loops that shrink to \(S\); note that these two conditions determine \(\alpha\) uniquely up to gauge equivalence. 

We consider the metric
\begin{equation}
g_{RF}= f g_{\beta} + f^{-1} \alpha^2 .
\end{equation}
It is clear that \(g_{RF}\) is locally hyperk\"ahler. On the other hand we can extend \(P_0\) to \( P = P_0 \sqcup(\mathbf{R} \times S^1) \sqcup \{\tilde{p} \}  \) so that the \(S^1\)-action extends smoothly to \(P\), acting freely on \( P \setminus \{ \tilde{p} \} \) and fixing \(\tilde{p} \). The map \(\Pi\) also extends smoothly as the orbit projection \(\Pi : P \to \mathbf{R}^3 \) with \( \Pi (\tilde{p}) = p \). We shall see that \(g_{RF}\) extends smoothly over \(\tilde{p}\) and as a metric with cone singularities along \( \Pi^{-1} (S) \cong \mathbf{R} \times S^1 \).

\subsection{Complex structure} \label{complex structure section}

Let \(\mathbf{R}_{\ast}^2\) be the \(re^{i\theta}\)-plane with the point \((1,0)\) removed. Define the 1-form on  \(\mathbf{R}_{\ast}^2 \times \mathbf{R} \) 

\begin{equation}
\alpha_0 = a_1 dr + a_2 \beta r d\theta
\end{equation}
where

\begin{equation} \label{connection}
a_1 = -\frac{1}{\beta r} \int_0^s \frac {\partial f}{\partial \theta} (r, \theta, q) dq, \hspace{4mm} a_2 =  \int_0^s \frac{\partial f}{\partial r} (r, \theta, q) dq .
\end{equation} 
It is clear that $\alpha_0$ is smooth on $(\mathbf{R}_{\ast}^2 \times \mathbb{R})\setminus S$; and the functions $a_1, a_2$ extend continuously by $0$ over $S$ as $C^{\alpha}$ functions for $\alpha = \beta^{-1} -1$. The computations that follow are done over the complement of $S$.

\begin{claim}
	
	\begin{equation}
	d\alpha_0 = - \star_{\beta} df
	\end{equation}
	
\end{claim}

\begin{proof}
	
	We have that
	
	$$ d \alpha_0 = -\frac{\partial a_2}{\partial s} \beta r d\theta ds + \frac{\partial a_1}{\partial s} ds dr + \left( \frac{\partial a_2}{\partial r} + \frac{1}{r}a_2 - \frac{1}{\beta r}\frac{\partial a_1}{\partial \theta} \right) \beta r dr d\theta$$
	and
	
	$$ \star_{\beta} df  = \frac{\partial f}{\partial r} \beta r d\theta ds + \frac{1}{\beta r} \frac{\partial f}{\partial \theta} ds dr + \frac{\partial f}{\partial s}  \beta r dr d\theta$$
	It is clear from \ref{connection} that ${\partial a_2}/{\partial s} = {\partial f}/{\partial r}$ and  ${\partial a_1} /{\partial s} =- ({1}/{\beta r}) {\partial f}/{\partial \theta}$. 
	
	It follows by symmetry that ${\partial f}/{\partial s} =0$ when $s=0$, so that $ {\partial f}/{\partial s} = \int_0^s \frac{\partial^2 f}{\partial s^2}(., t) dt$. Using that $\triangle_{\beta} f =0$ we get
	
	$$ \frac{\partial a_2}{\partial r} + \frac{1}{r}a_2 - \frac{1}{\beta r}\frac{\partial a_1}{\partial \theta} = \int_{0}^{s} \left(  \frac{\partial^2 f}{\partial r^2} + \frac{1}{r} \frac{\partial f}{\partial r} + \frac{1}{\beta^2 r^2} \frac{\partial^2 f}{\partial \theta^2} \right) (r, \theta, q) dq = - \int_0^s \frac{\partial^2 f}{\partial s^2}(r, \theta, q) dq = -   \frac{\partial f}{\partial s} $$
	
\end{proof}

Consider the product $\mathbf{R}_{\ast}^2 \times \mathbf{R} \times S^1$ and write points in the circle factor as $e^{it}$. Define the connection 1-form form $\alpha = dt + \alpha_0$ and the metric

\begin{equation}
g_{RF} = f g_{\beta} + f^{-1} \alpha^2
\end{equation}
The horizontal lifts of $\partial/\partial r, \partial/\partial \theta, \partial/\partial s$ are 

$$ \tilde{\frac{\partial}{\partial r}}=  \frac{\partial}{\partial r} - a_1  \frac{\partial}{\partial t}, \hspace{3mm} \tilde{\frac{\partial}{\partial \theta}}=  \frac{\partial}{\partial \theta} - a_2 \beta r  \frac{\partial}{\partial t} , \hspace{3mm} \tilde{\frac{\partial}{\partial s}}= \frac{\partial}{\partial s}$$
We consider the complex structure determined by the \(s\)-axis

\begin{equation}
I  \tilde{\frac{\partial}{\partial r}} = \frac{1}{\beta r}  \tilde{\frac{\partial}{\partial \theta}}, \hspace{3mm} I \frac{\partial}{\partial s} = -f  \frac{\partial}{\partial t} .
\end{equation}
The associated 2-form is
\begin{equation}
\omega_{RF} = g_{RF} (I . , .) =  \alpha  ds + f \beta r dr d\theta
\end{equation}

\begin{claim}
	$(g_{RF}, \omega_{RF}, I)$ defines a K\"ahler structure on $\mathbf{R}_{\ast}^2 \times \mathbf{R} \times S^1$  
\end{claim}

\begin{proof}
	The equation $d\alpha = - \star_{\beta} df$ implies that $d\alpha ds = - \frac{\partial f}{\partial s} \beta r drd\theta ds$ and this gives $d\omega = 0$. To prove that $I$ is integrable one can check that
	$$ \left[   \tilde{\frac{\partial}{\partial r}} + i \frac{1}{\beta r}  \tilde{\frac{\partial}{\partial \theta}},  \frac{\partial}{\partial s}  -i f  \frac{\partial}{\partial t} \right] =0 $$
	Strictly speaking we don't need to do this since we are going to find complex coordinates in what follows.
	
\end{proof}

The Cauchy-Riemann equations for a function $h$ to be holomorphic w.r.t. $I$ are given by

\begin{equation} \label{CR}
\frac{\partial h}{\partial r} + i \frac{1}{\beta r}  \frac{\partial h}{\partial \theta} = (a_1 + i  a_2) \frac{\partial h}{\partial t}, \hspace{4mm} \frac{\partial h}{\partial s}= if \frac{\partial h}{\partial t} .
\end{equation}
We look for a function $h$ which has weight one for the circle action,  this goes as in Lemma \ref{hol function lemma}. Let $h_0$ be any solution of the equation

\begin{equation} \label{h2}
\frac{\partial h_0}{\partial r} + i \frac{1}{\beta r}  \frac{\partial h_0}{\partial \theta} = 0 ,
\end{equation}
that is \(h_0\) is a holomorphic function of the variable $r^{1/\beta} e^{i\theta}$. Set

\begin{equation} \label{h1}
h= h_0 e^{-u} e^{it}, \hspace{3mm} u = \int_0^s f(r, \theta, q) dq
\end{equation}

\begin{claim}
	The function $h$ defined by \ref{h1} and \ref{h2} solves \ref{CR}. Similarly, holomorphic functions with weight $-1$ for the circle action are given by $h_0 e^u e^{-it}$.
\end{claim}

\begin{proof}
	
	$$\frac{\partial h}{\partial r} + i \frac{1}{\beta r}  \frac{\partial h}{\partial \theta} = e^{-u}e^{it} \left( \frac{\partial h_0}{\partial r} + i \frac{1}{\beta r}  \frac{\partial h_0}{\partial \theta} -  h_0 \frac{\partial u}{\partial r} - h_0   \frac{i}{\beta r}  \frac{\partial u}{\partial \theta} \right) = i h \left(- \frac{1}{\beta r} \frac{\partial u}{\partial \theta} + i \frac{\partial u}{\partial r} \right)$$ 
	and $ - \frac{1}{\beta r} \frac{\partial u}{\partial \theta} + i \frac{\partial u}{\partial r} = a_1 + i a_2$.
\end{proof}

Consider the segment in the \(re^{i\theta}\)-plane given by the points in the real line which are $\geq 1$, let $U$ be the  complement of that segment and \(U^{\ast}= U \setminus \{0\} \).
Write $c=\beta^{-1}$. The function $1-r^c e^{i\theta}$ maps $U$ to the complement of the negative real axis, so  
\begin{equation}\label{h0}
h_0 = (1-r^c e^{i\theta})^{1/2}
\end{equation}
is a well defined function un $U$ which satisfies \ref{h2}. From now on we set $h_0$ to be given by \ref{h0} and define
\begin{equation}
z= h_0 e^{-u}e^{i\psi}, \hspace{4mm} w=h_0 e^u e^{-i\psi}
\end{equation}

Let \(V \subset \mathbf{C}^2 \) be the open set of points \((z, w)\) such that \(zw\notin \mathbf{R}_{\leq 0} \). Write \(C=\{zw =1\}\).
\begin{claim}
	The map \(H =(z, w)\) gives a biholomorphism between \( (U^{*} \times \mathbf{R} \times S^1, I) \) and \(W \setminus C \). Moreover \(H\) extends as an homeomorphism between  \( (U \times \mathbf{R} \times S^1, I) \) with \(H (\{0\} \times \mathbf{R} \times S^1 ) = C \).
\end{claim}

\begin{proof}
	First we provide an inverse for \(H\), showing the homeomorphism part. The pair \((r, \theta)\) is determined by \( r^c e^{i\theta}= 1- zw \). The function \(u(s)=\int_{0}^{s} f(r, \theta, q) dq \) is increasing \(u' =f >0\) and \(\lim_{s \to \pm \infty} u(s)= \pm \infty\); so \((s, e^{it})\) is given by \(e^{-u}e^{it}=h_0^{-1}z\).
	
	Next we compute the Jacobian of the map \(H\). Let \(\eta_1 = dr + i \beta r d\theta \) and \( \eta_2 = ds - i f^{-1}\alpha \), so that \(\{\eta_1, \eta_2\}\) is a basis of the \((1, 0)\)-forms on \( U^{*} \times \mathbf{R} \times S^1 \). It is straightforward to check that
	\[ dz = z \left( h_0^{-1} \frac{\partial h_0}{\partial r} - a_2 - i a_1 \right) \eta_1 - zf \eta_2 \] 
	\[ dw = w \left( h_0^{-1} \frac{\partial h_0}{\partial r} + a_2 + i a_1 \right) \eta_1 + wf \eta_2 . \] 
	The determinant of the linear map that takes \(\{\eta_1, \eta_2\}\) to \(\{dz, dw\}\) is \(-fcr^{c-1}e^{i\theta}\) and is non-zero on \( U^{*} \times \mathbf{R} \times S^1 \). 	
\end{proof}

We can compose the inverse of \(H\) with the projection of the trivial \(S^1\)-bundle \( \mbox{pr} (re^{i\theta}, s, e^{it}) = (re^{i\theta}, s) \) to obtain the map \(\Pi = \mbox{pr} \circ H^{-1} : V \to \mathbf{R}^3 \). We want to show that \(\Pi\) extends to all of \(\mathbf{C}^2\), as an orbit map for the \(S^1\)-action \(e^{it}(z, w)= (e^{it}z, e^{-it}w)\). Clearly the \(r, \theta\) coordinates of \(\Pi\) extend, since \(r^c e^{i\theta} = 1-zw\). The key step is to extend the function \(s\).

\begin{claim}
	\(s \) extends to \(\mathbf{C}^2\), smoothly on the complement of \(C\). The map \(\Pi : \mathbf{C}^2 \to \mathbf{R}^3 \) is an orbit projection for the \(S^1\)-action \(e^{it}(z, w)= (e^{it}z, e^{-it}w)\) with \(\Pi(0)=p\) and \(\Pi(C)= \{0\}\times\mathbf{R}\).
\end{claim}

\begin{proof}
	The coordinate \(s\) is determined by \(e^{-u}e^{it}=h_0^{-1}z\), since \(|h_0|=|z|^{1/2}|w|^{1/2}\) we obtain \(e^{-u} =|z|^{1/2} |w|^{-1/2} \) and taking logarithms
	\begin{equation} \label{function s}
		\int_{0}^{s} f(re^{i\theta}, q)dq = \frac{1}{2} \log \left( \frac{|z|}{|w|} \right) .
	\end{equation}
	It is then clear that in the complement of \(\{zw = 0\} \) the map \(\Pi\) extends with the desired properties. We assume that \(|z w| < \epsilon\) for some small \(\epsilon\). Since \(r^c e^{i\theta} = 1-zw\), we can suppose that \(-\pi<\theta<\pi\). Let \(\tilde{\theta}=\beta \theta \), so that 
	\begin{equation} \label{greens}
		 f(re^{i\theta}, q) = \frac{1}{2 \sqrt{|re^{i\tilde{\theta}}-1|^2 +q^2}} +\frac{F}{2}
	\end{equation}
	for some smooth harmonic \emph{positive} function \(F= F(re^{i\tilde{\theta}}, q)\) -see Lemma \ref{positive lemma}-. Write \( \xi = re^{i\tilde{\theta}} -1 = (1-zw)^{\beta}-1 = \beta zw \psi \) for some \(\psi\) holomorphic function of \(zw\) with \(\psi(0)=1 \). We plug \ref{greens} into \ref{function s} to obtain
	\[ \log \left( \frac{s+\sqrt{s^2 + |\xi|^2}}{|\xi|} \right) + \int_{0}^{s} F = \log \left( \frac{|w|}{|z|} \right) . \]
	We exponentiate and re-arrange terms to obtain
	\begin{equation} \label{global s}
		2s= \beta |w|^2 \psi e^{-\int_{0}^{s}F} - \beta |z|^2 \psi e^{\int_{0}^{s}F} ;
	\end{equation}
	note that in the standard case of \(\beta=1\), we have \(F \equiv 0\), \(\psi \equiv 1\) and therefore \(2s = |w|^2 - |z|^2\) -see \ref{Hopf map}-. The claim follows from \ref{global s}. The map \(\Pi\) sends the \(\{z=0\}\) complex line to the ray \(\{(1, 0, s): s>0\}\) via \(2s e^{\int_{0}^{s}F(1, 0, q)dq} = \beta |w|^2\), the fact that \(F>0\) implies that \(s \to se^{\int_{0}^{s}F}\) is a diffeomorphism of \([0, \infty)\). Similarly, \(\Pi\) sends \(\{w=0\}\) to \(\{(1, 0, s): s<0\}\) via \(2s e^{\int_{s}^{0}F(1, 0, q)dq} =- \beta |z|^2\).
	
\end{proof}

\begin{figure}
	\centering
	\includegraphics[width=0.7\linewidth]{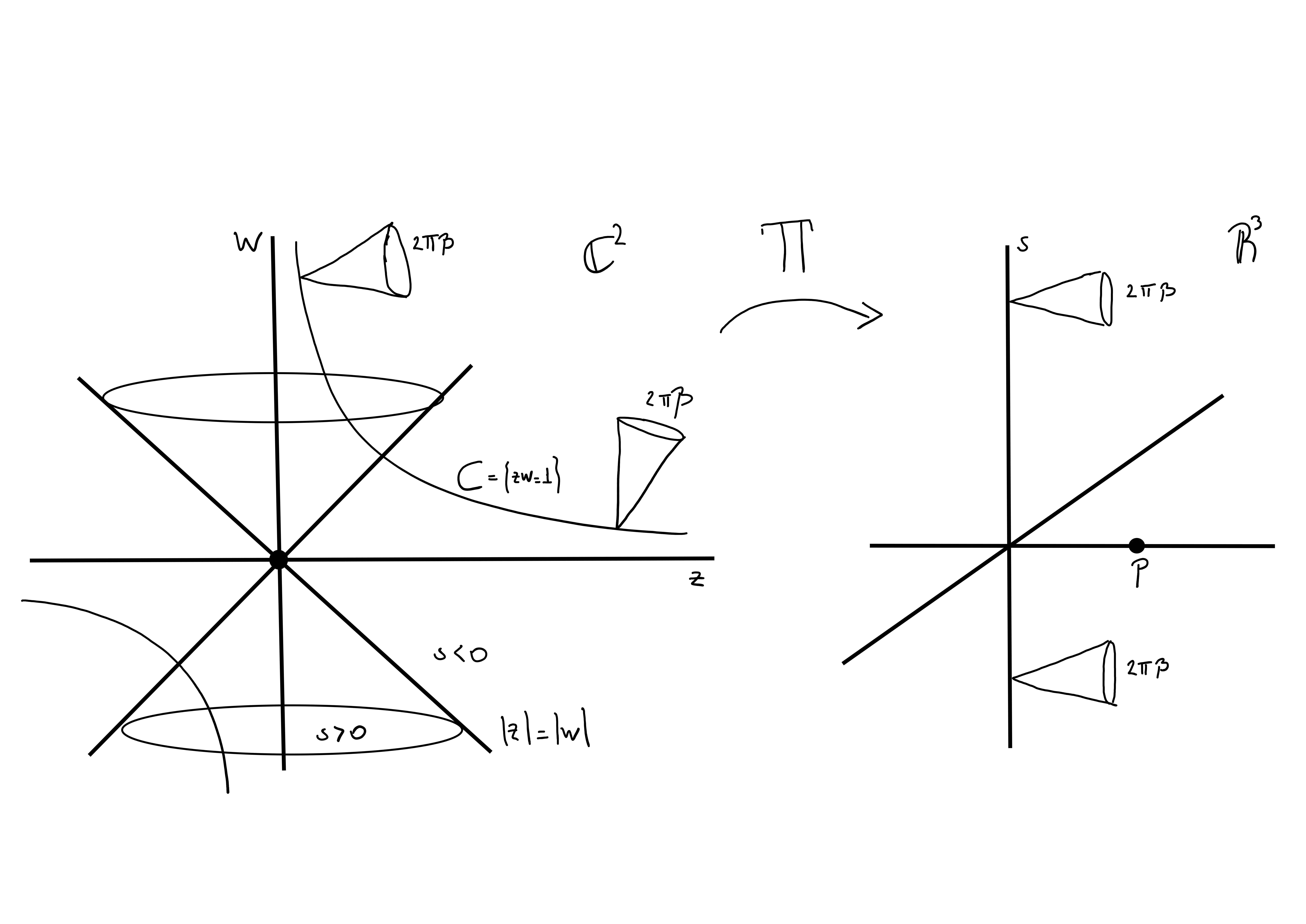}
	\caption{}
	\label{fig:bundleprojection}
\end{figure}

\(I\) is  the standard complex structure in the \(z, w\) coordinates. Since \(I ds = f^{-1}\alpha\), it follows that \(\alpha\) extends as a  connection \(1\)-form on \(\mathbf{C}^2 \setminus \{0\}\) which is smooth on the complement of the conic. It is then standard to show -see for example \cite{AKL}- that \(g_{RF}\) extends smoothly over \(0\), since \(g\) is Euclidean in a neighborhood of \(p\), \(f\) differs from the Newtonian potential by a smooth harmonic function and \(\alpha\) is a smooth connection in a punctured neighborhood of \(0\). The function \(s\) is a moment map for the circle action, in the sense that \(\omega_{RF}(Y, \cdot)= ds \) where \(Y=2 \mbox{Im} \left( w \partial/ \partial w - z \partial / \partial z \right)\). \(s>0\) when \(|w|>|z|\) and \(s<0\) when \(|z|>|w|\). The metric \(g_{RF}\) is invariant under the map \((z, w) \to (w, z) \) which lifts \((re^{i\theta}, s) \to (re^{i\theta}, -s)\). See Figure \ref{fig:bundle-projection-1}
 
The statement about the cone singularities along \(C\) also follows easily. Let \(p \in C\), we take holomorphic coordinates \(z_1 = zw -1\) and \(z_2 = w\), say, so \(C= \{z_1=0\}\). Write \(z_1=r^c e^{i\theta}\) and let \(\epsilon=dr+i\beta rd\theta\), we want to show that \(\omega_{RF}\) has \(C^{\alpha}\) coefficients with repect to the basis \(\{\epsilon \overline{\epsilon}, \epsilon \overline{dz_2}, dz_2 \overline{\epsilon}, dz_2 \overline{dz_2} \}\). Comparing with our previous notation we see that \(\epsilon = \eta_1 \) and 
\[\omega = (if/2) (\epsilon \overline{\epsilon} + \eta_2 \overline{\eta_2} ) , \] 
where \( \eta_2 = \gamma_1 dz_2 - \gamma_2 \epsilon \) with 
\[ \gamma_1 = (w f)^{-1}, \hspace{3mm} \gamma_2 = f^{-1} \left( h_0^{-1} \frac{\partial h_0}{\partial r} + a_2 + i a_1 \right) . \]
The first item in Theorem \ref{THEOREM} then follows from \(h_0^{-1} \partial h_0 / \partial r = (1/zw) c r^{c-1} e^{i\theta} \) and the formula \ref{connection} for \(a_1, a_2\).

To conclude this first part we compute the volume form of \(g_{RF}\). We write \( \Omega = (1/\sqrt{2}) dz dw \).
\begin{claim}
	\begin{equation} \label{volume form}
		\omega_{RF}^2 = \beta^2 |1-zw|^{2\beta-2} \Omega \wedge \overline{\Omega}
	\end{equation}
\end{claim}
\begin{proof}
	It is immediate from the previous computation of the Jacobian of the map \(H=(z, w)\) that
	\[ \Omega \wedge \overline{\Omega} = (1/2)f^2c^2r^{2c-2}\eta_1 \eta_2 \overline{\eta_1 \eta_2}. \]
	We use that 
	\[ \eta_1 \overline{\eta_1} = -2i\beta r dr d\theta, \hspace{2mm} \eta_2 \overline{\eta_2} = 2i f^{-1}ds \alpha, \hspace{2mm} \omega^2=2f\beta r \alpha ds dr d\theta \]
	and \( r= |1 -zw|^{\beta} \) to conclude the claim.
\end{proof}

\subsection{Asymptotics} \label{asymptotics section}

 There is a simple explanation, in terms of the complex curve \(C\), for the exponent \(-2 /\beta\) in Item \ref{Item 2} of Theorem \ref{THEOREM}. Among the diffeomorphisms \(F\) of \( \mathbf{C}^2 \) which, outside a compact set, take the conic \(C=\{zw=1\}\) to its asymptotic lines \(\{zw=0\}\), the ones which are closest to being holomorphic  satisfy \( |\overline{\partial} F (x)| = O (|x|^{-2}) \). On the other hand  \(\rho^2 = |z|^{2\beta}+ |w|^{2\beta}\) and therefore \( |\overline{\partial} F (x)| = O (\rho^{-2/ \beta})\). Our proof of Item \ref{Item 2} is based on the simple observation that applying the Gibbons-Hawking ansatz to \(2 \pi \Gamma_0\) gives rise to  \( \mathbf{C}_{\beta} \times \mathbf{C}_{\beta} \). The asymptotic behavior of \(g_{RF} \) then follows from the fact that \(\Gamma_p \) is asymptotic to \(\Gamma_0\). The `matching map' \(\Phi \) is as a suitable bundle map.

We begin by writing \(g_F = \beta^2 |u|^{2\beta-2} |du|^2 + \beta^2 |v|^{2\beta -2} |dv|^2\) as a Gibbons-Hawking metric. We use the cone coordinates \(u= \rho_1^{1/\beta}e^{i\psi_1}\) and \(v= \rho_2^{1/\beta}e^{i\psi_2}\) so that
\[ g_F = d\rho_1^2 + \beta^2 \rho_1^2 d\psi_1^2 + d\rho_2^2 + \beta^2 \rho_2^2 d\psi_2^2 . \]
Let \(\Pi_0 : \mathbf{C}^2 \to \mathbf{R}^3 \) be defined as
\[ \Pi_0 (u, v) = \left( \beta \rho_1 \rho_2 e^{i (\psi_1 + \psi_2)}, \beta \frac{\rho_2^2 - \rho_1^2}{2} \right) . \]
This is an orbit map for the \(S^1\)-action \(e^{it}(u, v)=(e^{it}u, e^{-it}v)\). If we let \( x= \Pi_0 (u, v) \) then \( |x|^2 = (\beta^2 / 4)(\rho_1^2 +\rho_2^2)^2 \), equivalently
\begin{equation} \label{distance relation}
	\beta \rho^2 = 2|x| .
\end{equation}
The derivative of the action is \(Y = \partial / \partial \psi_1 - \partial / \partial \psi_2 \) and \(|Y|^2_{g_F}= \beta^2 \rho^2 = 1 / (2\beta|x|) \). We let \( \alpha_0 = |Y|_{g_F}^{-2} g_F(Y, \cdot) = \rho^{-2} (\rho_1^2 d\psi_1 - \rho_2^2 d\psi_2) \). It requires a simple computation to check that
\begin{equation}
	g_F = f_0  g + f_0^{-1} \alpha_0^2, \hspace{3mm} \mbox{with} \hspace{1mm} f_0 = \frac{1}{2 \beta |x|} .
\end{equation} 
Note that \(f_0 = 2\pi \Gamma_0 \).

Now let \((\mathbf{C}^2, g_{RF}) \) together with \(\Pi : \mathbf{C}^2 \to \mathbf{R}^3\). We let \( B \subset \mathbf{R}^3 \) a closed ball of radius \(2\), say, so that \(p \in B \). The bundles \(\Pi_0\) and \(\Pi\) are isomorphic on the complement of \(B\), so there is an \(S^1\)-equivariant diffeomorphism \(\Phi : \Pi_0^{-1} (\mathbf{R}^3\setminus B ) \to \Pi^{-1} (\mathbf{R}^3\setminus B )\) which induces the identity on \(\mathbf{R}^3 \setminus B\), in particular note that \(\Phi (\{uv =0\}) \subset \{zw =1\} \). 

\(\Phi_0^{\ast} \alpha \) is a connection on \(\Pi_0\), therefore \( \Phi_0^{\ast} \alpha - \alpha_0 = \eta \) with \(\eta \) a 1-form on \(\mathbf{R}^3 \setminus B \). Moreover
\[ d \eta = -2\pi \star_{\beta} d (\Gamma_p - \Gamma_0) . \]
On the other hand, since \( |\Gamma_p - \Gamma_0| = O(|x|^{-1-1/\beta}) \), we can assume -after changing gauge if necessary- that \(|\alpha - \alpha_0| = O (|x|^{-1-1/\beta}) \). We evaluate \(g_{RF}\) in the orthonormal basis of \(g_F\) given by \(v_1= f_0^{1/2} Y\) and \(v_2, v_3, v_4 \) the horizontal lifts of \( f_0^{-1/2} \partial/ \partial r, f_0^{-1/2} (\beta r)^{-1} \partial/ \partial \theta, f_0^{-1/2} \partial/ \partial s \); our goal is to show that \(g_F (v_i, v_j) = \delta_{ij} + O(|x|^{-1 / \beta}) \) which is the same -due to \ref{distance relation}- as \(|g_{RF}-g_F|_{g_F} = O(\rho^{-2/\beta}) \). Note that \( \alpha (v_1)= f_0^{1/2} \) and \(\alpha (v_j)= f_0^{-1/2} O(|x|^{-1-1/\beta}) \)  for \(j=2, 3, 4\). This follows from straightforward computation, but before doing that we state a simple observation.
\begin{lemma}
	Let \(f= O(|x|^{-a}) \) and \( f-g = O(|x|^{-b}) \) with \(0 < a< b\) and \(f >0\). Then \(g/f = 1 + O(|x|^{-(b-a)})\).
\end{lemma} 
Indeed
\[ g/f = f/f + (g-f)/f . \]
In particular, we conclude that \(f /f_0 = 1 + O (|x|^{-1/\beta}) \). 

We proceed with our proof, we let \(2 \leq j, k \leq 4\) and \(j \ne k\)
\[g_{RF}(v_1, v_1)= f^{-1}f_0 = 1 + O(|x|^{-1/\beta}), \hspace{2mm} g_{RF}(v_j, v_j)= f_0^{-1} f + O(|x|^{-2/\beta}) = 1 + O(|x|^{-1/\beta})  \]

\[g_{RF}(v_1, v_j)= f^{-1} O(|x|^{-1-1/\beta})= O(|x|^{-1/\beta}) , \hspace{2mm} g_{RF}(v_j, v_k)= O(|x|^{-2/\beta}) . \]
Similarly, we can show that \( |\Phi^{\ast}\omega_{RF} - \omega_F|_{g_F} = O(\rho^{-2/\beta}) \) and therefore \( |\Phi^{\ast}I - I|_{g_F} = O(\rho^{-2/\beta}) \).

\begin{remark}
	We can include derivatives in the statement of Item \ref{Item 2}, as \(| \nabla_X (\Phi^{*}g_{RF}-g_F)|_{g_F} = O(\rho^{-2/\beta -1}) \) and so on; but care must be taken in not to differentiate in transverse directions to the cone singularities more than once. 
\end{remark}

\subsection{Energy} \label{energy section}

It follows from \ref{curv op GH met} that the curvature operator of \(g_{RF}\) is given, up to the \(-f/2\) factor, by the trace free part of the Hessian of \(f^{-2}\) -with respect to the \(g_{\beta}\) metric-. In particular, close to the conic the curvature behaves as \(r^{1/\beta-2}\) and this is unbounded when \(\beta>1/2\). The norm-square of the curvature operator is \(O(r^{2/\beta-4})\). Comparison with the integral \(\int_{0}^{1} r^{2/ \beta-3}dr < \infty \) shows that \(|\mbox{Rm}(g_{RF})|^2 \) is locally integrable.

According to our formula \ref{energy distribution}, 
\[ |\mbox{Rm}(g_{RF})|^2 = \frac{1}{4f} \Delta_{\beta} \Delta_{\beta} f^{-1} .\]
We want to compute \(\int_{\mathbf{C}^2} |\mbox{Rm}(g_{RF})|^2\).
We note that \( \Pi: (\mathbf{C}^2, g_{RF}) \to (\mathbf{R}^3, f \cdot g)\) is a Riemannian submersion whose fiber over \(x\) is a circle of length \( 2\pi f^{-1/2}(x)\). The volume form of  \( f \cdot g\) is \(f^{3/2}dV_{\beta}\) and it is easy to conclude that
\begin{equation} \label{en proof 1}
	 \| \mbox{Rm}(g)\|_{L^2}^2 = (\pi /2) \int_{\mathbf{R}^3} \Delta_{\beta} \Delta_{\beta} f^{-1} dV_{\beta} . 
\end{equation}
In order to compute this quantity we use Stokes' theorem 
\[ \int_{\Omega} \Delta_{\beta} \Delta_{\beta} f^{-1} dV_{\beta} = \int_{\partial \Omega} \langle D \Delta_{\beta} f^{-1}, \nu \rangle dA_{\beta}  \] 
for an increasing sequence of domains \(\Omega\). 

There are two key lemmas

\begin{lemma}
	Let \(C_r\) be a bounded cylinder consisting of  points which are at distance \(r\) from the singular set \(S= \{0\} \times \mathbf{R} \). Then  
	\[ \lim_{r \to 0} \int_{C_r} \langle D \Delta_{\beta} f^{-1}, \nu \rangle dA_{\beta} =0 \]
\end{lemma}

\begin{proof}
	The lemma is a consequence of the \(\beta\)-smoothness of \(f^{-1}\) together with the fact that \(\triangle_{\beta} f =0 \). Indeed, since \(f\) is harmonic, \( \triangle_{\beta} f^{-1} = 2 f^{-3} |Df|^2 \). The \(\beta\)-smoothness then gives us \(|D \Delta_{\beta} f^{-1}| = O (r^{2\beta^{-1} -3}) \) and therefore \(\int_{C_r} \langle D \Delta_{\beta} f^{-1}, \nu \rangle dA_{\beta} =  O (r^{2\beta^{-1} -2}) \).
\end{proof}

\begin{lemma}
Let \(S_R\) denote the sphere of points which are at distance \(R\) from \(0\). Then 
\[ \lim_{R \to \infty} \int_{S_R} \langle D \Delta_{\beta} (f^{-1} -f_0^{-1}), \nu \rangle dA_{\beta} =0 \]	
\end{lemma}

\begin{proof}
	
	\[\Delta_{\beta} (f^{-1}-f_0^{-1}) = f^{-3}|Df|^2 - f_0^{-3}|Df_0|^2 = f_0^{-3} |Df_0|^2 \left( (f/f_0)^{-3} (|Df|/|Df_0|)^2 -1 \right)  \]
	\(f_0^{-3}|Df_0|^2=O(|x|^{-1}) \) and \(f/f_0 = 1 + O(|x|^{-1/\beta}) \). On the other hand, \( | |Df|^2 - |Df_0|^2 | \leq (|Df| + |Df_0|) | D(f-f_0)| \) implies that \(|Df|^2/|Df_0|^2 = O (|x|^{1-1/\beta})\). We conclude that \( \Delta_{\beta}(f^{-1}-f_0^{-1}) = O(|x|^{-1/\beta}) \) 
	
	Note that \(\nu\) is tangential to \(S\), so that $\langle D \Delta_{\beta} (f^{-1} -f_0^{-1}), \nu \rangle = O(|x|^{-1/\beta -1}) $. We deduce that the integral is \(O(|x|^{1-1/\beta})\).

\end{proof}

It follows easily from these results that 
\begin{equation} \label{en proof 2}
	 \int_{\mathbf{R}^3} \Delta_{\beta} \Delta_{\beta} f^{-1} dV_{\beta} = \lim_{R \to \infty} \int_{S_R(0)} \langle D \Delta_{\beta} f_0^{-1}, \nu \rangle dA_{\beta} - \lim_{\epsilon \to 0} \int_{S_{\epsilon}(p)} \langle D \Delta_{\beta} f^{-1}, \nu \rangle dA_{\beta} .
\end{equation}
Finally, 

\begin{itemize}
	\item \[\lim_{\epsilon \to 0} \int_{S_{\epsilon}(p)} \langle D \Delta_{\beta} f^{-1}, \nu \rangle dA_{\beta} = -16 \pi .\] 
	Indeed, \( g\) is isometric to the Euclidean metric in a neighborhood of \(p\) and we reduce to the standard situation where \(f= 1/2|x| \). We compute in spherical coordinates to obtain \(\Delta |x| = 2 |x|^{-1} \) and \( \int_{S} \langle D(1/|x|), \nu \rangle dA= -4\pi, \)  for any sphere \(S\) centered at \(0\). 
	
	\item \[\lim_{R \to \infty} \int_{S_R(0)} \langle D \Delta_{\beta} f_0^{-1}, \nu \rangle dA_{\beta} = -\beta^2 16 \pi .\]
	This goes along the same lines as in the previous item, replacing the Euclidean metric with \(g\). The extra factor \(\beta^2\) comes from \(f_0^{-1}= 2 \beta |x| \) and \(dA_{\beta}= \beta dA\).
\end{itemize}

We put \ref{en proof 1} and \ref{en proof 2} together, to obtain our desired formula

\begin{equation}
	\| \mbox{Rm}(g_{RF}) \|^2_{L^2} = 8 \pi^2 (1-\beta^2) .
\end{equation}

\section{Additional comments} \label{additiona comments}

\subsection*{Curvature and directions}
As we mentioned, the curvature of \(g_{RF}\) is unbounded at points of the conic when \(\beta>1/2\); and  it is conjectured that this is the case for any K\"ahler-Einstein metric. We can also ask in which directions the curvature blows-up. We can be more precise with the concept of direction at points of the curve \(C\). We consider the \(\mathbf{CP}^1\)-bundle of directions, \(P\), over \(\mathbf{C}^2\); whose fiber over \(x\) is \(\mathbf{P}(T_x \mathbf{C}^2)\). The metric \(g_{RF}\) is smooth in the complement of \(C\) and taking orthogonal complements defines an automorphism, \(\perp\), of \(P\) over that region. The point is that \(\perp\) extends continuously over all of \(P\) and therefore there is a well-defined notion of a normal direction to the curve \(C\). 

On the other hand it is a general fact that on Einstein \(4\)-manifolds the sectional curvatures of mutually orthogonal planes agree, indeed the Einstein condition is equivalent to the commutativity of the  curvature operator with the Hodge star. We can provide a simple proof for the Ricci-flat K\"ahler case: Take normal coordinates at \(p\), so that \(g_{1\overline{1}}g_{2\overline{2}} - |g_{1\overline{2}}|^2 = e^F \) for some pluri-harmonic function \(F\) whose gradient vanishes at \(p\). Differentiate with respect to \(z_1\), \(\overline{z_1}\) and evaluate at \(p\) to obtain \( g_{1\overline{1}, 1\overline{1}} + g_{2\overline{2}, 1\overline{1}} =0\); similarly differentiating with respect to \(z_2,\overline{z_2}\) we obtain \( g_{1\overline{1}, 2\overline{2}} + g_{2\overline{2}, 2\overline{2}} =0\). It follows that \(  g_{1\overline{1}, 1\overline{1}} = g_{2\overline{2}, 2\overline{2}} \); which is to say that the sectional curvature of the \(\partial/ \partial z_1 \)
and \(\partial/ \partial z_2 \) planes at \(p\) agree.

We go back to our setting, \(g_{RF}\) is smooth in tangent directions to \(C\) and it induces on it the metric of a rotationally symmetric negatively curved cylinder. The sectional curvature of \(g_{RF}\) remains bounded as we approach the curve in either tangential or normal directions. The upshot is that if \(p \in C \) and \(P_p \cong \mathbf{CP}^1 \) with \(\perp\)  the standard \(\xi \to -1/\overline{\xi} \) and the tangent and normal directions to the curve corresponding to the North and South poles; then the sectional curvature is invariant under \(\perp\), bounded around the poles and unbounded around the equator.

\subsection*{Quotients and limits when \( \beta \to 0 \)} \footnote{I want to thank Hans-Joachim Hein for discussions related to the content of this section.} 
We study the case when \(\beta=1/n\) with \(n\) a positive integer \(\geq 2\). The Green's function is explicit in this case, given by the Neumann reflection trick
\[ \Gamma_p (x) = \frac{1}{4\pi} \sum_{j=0}^{n-1} \frac{1}{|x-p_j|}, \]
with \(x= (r e^{i\tilde{\theta}}, s)\) and \(p_j=(e^{2\pi i(j/n)}, 0)\).

Let \(X\) be the standard \(A_{n-1}\)-ALE space determined by \(p_0, \ldots, p_{n-1}\). The orthogonal transformation that fixes the \(s\)-axis and rotates by \(\pi/n\) has a lift as an isometry of \(X\), this lift is unique if we require that it fixes the points over the axis. This isometry generates a \(\mathbf{Z}_n\)-action and the quotient space is then identified with \((\mathbf{C}^2, g_{RF})\). When \(n=2\), \(X\) is the Eguchi-Hanson space and there is an explicit K\"ahler potential for \(g_{RF}\) which -up to a constant factor- is given by \(\phi = (|z|^2 + |w|^2 + |1-zw| +1)^{1/2}\).

We let \(g_n \) denote the metric \(g_{RF}\) with \(\beta=1/n\). We want to know the possible Gromov-Hausdorff limits of the sequence \(\{g_n\}\) as \(n \to \infty\). First of all we note that \(g_n\) are complete, in the sense that Cauchy sequences with respect to the induced distance  converge. Indeed the standard proof for Gibbons-Hawking spaces applies in our case -see \cite{AKL}-, the point being that the bundle projection is a Riemannian submersion onto a complete space. Since we are dealing with non-compact spaces we must choose points and talk about pointed Gromov-Hausdorff limits. We choose the points to be the ones fixed by the \(S^1\)-action. As we shall see the curvature of \(g_n\) blows-up at this point and if we re-scale in order to keep it bounded we obtain the Taub-Nut metric in the limit, in symbols \( (\mathbf{C}^2, \lambda_n g_n, 0) \to (\mathbf{C}^2, g_{TN}, 0) \) with \(\lambda_n \approx |\mbox{Rm}(g_n)|\) and 
\[ 0< \lim_{n \to \infty} \frac{|\mbox{Rm}(g_n)|(0)}{n \log n} < \infty . \]

We consider a unit circle in the plane with \(n\) points equally separated. If we fix one of this points and consider the sum of the inverse distances to the others, then -up to a constant factor- the sum is \(n(1+1/2+\ldots+1/n)\). We go back to the sequence \(g_n\), with the marked points mapping to \(0 \in \mathbf{R}^3\) and conclude that in a small ball we can write the harmonic functions as \(1/2|x| + (n \log n) F_n \) with \(F_n\) converging uniformly to a positive constant. It is then easy to derive the claims made in the previous paragraph.

It is worth to point out that in the previous limit we are magnifying a neighborhood of \(0\) and pushing-off the cone singularities to infinity. On the other hand 
\[ \lim_{n \to \infty} \| \mbox{Rm}(g_{n}) \|_{L^2}^2 = \lim_{n \to \infty} 8\pi^2 (1-1/n^2)= 8 \pi^2 = \| \mbox{Rm}(g_{TN}) \|_{L^2}^2 .\]
So the metrics \(g_n\) become nearly flat, as \(n \to \infty\), around the conic. It is also tempting -by approximating large circles with a line- to compare the metrics \(g_n\) for \(n\) large with `the' Ooguri-Vafa metric \cite{GrossWilson}, obtained from the Gibbons-Hawking ansatz applied to the potential of infinitely many charges lying on a line and equally separated. However, as a word of caution, it must be said that the Ooguri-Vafa metric is \emph{not} complete. The Ooguri-Vafa metric is indeed a one parameter family of metrics, parametrized by the distance between the charges. Scaling the metrics when the parameter tends to zero at the fixed point of the \(S^1\)-action  recovers the Taub-Nut metric in the limit; we can use the triangle inequality to relate this sequence to \(\lambda_n g_n\). However one can ask for a more precise correspondence, relating their associated harmonic functions  -both admitting asymptotic expansions in terms of Bessel functions-.

We can also scale the metrics \(g_{RF}\) so that their volume forms are \(|1-zw|^{2\beta-2} \Omega \wedge \overline{\Omega}\) and then take the point-wise limit of these tensors as \(\beta \to 0\), proceeding like this results into a degenerate limit \(g_{\infty} \geq 0\). It is not known whether there is a K\"ahler metric on the complement of the conic with volume form  \(|1-zw|^{-2} \Omega \wedge \overline{\Omega}\). Note that such a metric would be necessarily complete Ricci-flat and that \(\pi_1(\mathbf{C}^2 \setminus C) \cong \mathbf{Z}\).    

\subsection*{Variants}
As mentioned in \cite{DonaldsonKMCS}, there are many variants of the construction. Finite sums of Green's functions \(\Gamma_p\) at different points give rise to Ricci-flat metrics with cone singularities on \(A_n\) manifolds. It is also possible to consider several parallel wedges an obtain metrics on \(\mathbf{C}^2\) with cone singularities along disjoint conics. Another variant is to add a positive constant  term to the Green's function to obtain analogs of (multi)-Taub-Nut spaces. 

More interesting is the case of a curve \(C \subset \mathbf{C}^2\) which is invariant under an \(S^1\)-action different from the one we considered. For example $\{w=z^2\}$ and \(\{wz^2=1\}\) are invariant under \((e^{it}z, e^{2it}w)\) and \((e^{it}z, e^{-2it}w)\) respectively. We can ask for \(S^1\)-invariant Ricci-flat metrics with cone singularities along the curve; but a suitable extension of the Gibbons-Hawking ansatz to the context of Seifert fibrations or \(S^1\)-actions which rotate the complex volume form seems not to be available. 

\bibliographystyle{plain}
\bibliography{ReferencesGH}

\end{document}